\documentclass[10pt]{article}
\usepackage{amsmath,amssymb, amsfonts, amsthm,amscd, dsfont}
\usepackage[T2A]{fontenc}
\usepackage[cp1251]{inputenc}
\usepackage[english]{babel}

\newtheorem{proposition}{Proposition}
\newtheorem{lemma}{Lemma}
\newtheorem{theorem}{Theorem}
\newtheorem{corollary}{Corollary}
\newtheorem*{proposition*}{Proposition}
\newtheorem*{lemma*}{Lemma}
\newtheorem*{theorem*}{Theorem}
\newtheorem*{corollary*}{Corollary}

\newtheorem*{definition*}{Definition}
\newtheorem*{remark*}{Remark}
\newtheorem*{example*}{Example}

\begin{document}

\begin{center}
{\LARGE On Leibniz  algebras with  maximal  cyclic  subalgebras}
\end{center}

\begin{center}
V.A.~Chupordia$^1$, L.A.~Kurdachenko$^1$, I.Ya.~Subbotin$^2$

$^1$ Department of Algebra and Geometry, Oles Honchar Dnipro National University, Gagarin prospect 72, Dnipro 10, 49010, Ukraine

$^2$ Mathematics Department, National University, 5245 Pacific Concourse Drive, Los Angeles, CA 90045, USA

\end{center}

\textbf{Abstract:} We begin to study the structure of Leibniz algebras having maximal cyclic subalgebras.


\textbf{Classification:} 17A32, 17A60, 17A99.

	\section{Introduction}

Let $L$ be an algebra over a field $F$ with the binary operations $+$ and $[-,-]$. Then $L$ is called a \textit{left  Leibniz  algebra} if it satisfies the left Leibniz identity

$$[[a,b], c] = [a, [b,c]] - [b,[a,c]]$$

for all $a, b, c \in L$. We will also use another form of this identity:

$$[a, [b, c]] = [[a, b], c]+[b, [a, c]].$$

Leibniz algebras first appeared in the paper of A.~Bloh~\cite{BA1965}, but the term "Leibniz algebra" appears in the book of J.L.~Loday~\cite{LJ1992} and in the article of J.L.~Loday~\cite{LJ1993}. In~\cite{LP1993} J.~Lodey and T.~Pirashvili began the systematic study of properties of Leibniz algebras. The theory of Leibniz algebras has developed very intensively in many different directions. Some of the results of this theory were presented in the book~\cite{AOR2020}. Note that the class of Lie algebras is a subclass of the class of Leibniz algebras. Conversely, if  $L$  is a Leibniz algebra, in which the identity $[a, a] = 0$ is valid for every element  $a \in L$, then it is a Lie algebra. The question about those properties of Leibniz algebras that Lie algebras do not have and, accordingly, about those types of Leibniz algebras that have essential differences from Lie algebras naturally arises. A lot has already been done in this direction. This can be illustrated, for example, by finite-dimensional Leibniz algebras. Their description and the structure of the already described finite-dimensional Leibniz algebras are much more complicated than the structure of Lie algebras of the corresponding dimension (see the book~\cite{AOR2020} and the paper~\cite{YaV2019}). As a rule, the description of finite-dimensional Leibniz algebras is reduced to finding their structural constants. However, knowledge of structural constants does not always give a relief picture of the structure of the corresponding finite-dimensional Leibniz algebra. Here you can refer to the experience of finite groups. In the theory of finite groups, the description of finite groups based on the specification of their definite properties has turned out to be much more fruitful than the description of groups of fixed orders. Therefore, it seems natural to use this approach in the theory of Leibniz algebras. Cyclic Leibniz algebras are among the simplest ones. Their description was obtained in~\cite{CKSu2017}. As a consequence of this description, a description of the "minimal Leibniz algebras," that is, Leibniz algebras, whose proper subalgebras are Lie algebras and Leibniz algebras, whose proper subalgebras are abelian was obtained. This approach turned out to be quite effective. We will not do any review of the results here, we just make links to surveys~\cite{KKPS2017,KSeSu2020}, and the papers~\cite{CKSe2020, KOP2016, KOS2019, KSeSu2017, KSeSu2018, KSuSe2018, KSuSe2018A,  KSY2020A, YaV2019}.

One of the first steps in the theory of finite groups was the study of groups which were close to abelian, and in particular, to cyclic groups. Finite groups, having a maximal cyclic subgroup, have been described long time ago (see, for example, the book~\cite[\S 1]{BeYa2008}). For Leibniz algebras, the question about the structure of a Leibniz algebra having a maximal cyclic subalgebra naturally arises. The study of such Leibniz algebras begins in this work. In contrast to finite groups, the situation here is much more varied. Note at least the circumstance that cyclic Leibniz algebras (with the exception of one-dimensional ones) are non-Abelian. The first case to be considered here is the case of nilpotent Leibniz algebra. Since infinite dimensional cyclic Leibniz algebra is not nilpotent, in our case a Leibniz algebra must has finite dimension. First our main result gives a description of such algebras.

\begin{theorem}\label{therem A}
Let  $L$  be a nilpotent Leibniz algebra of finite dimension $n + 1 \geq 3$ over a field  $F$. Suppose that  $L$  includes a maximal cyclic subalgebra  $K$. Then  $L$  is an algebra of one of the following types: 

(i)  $L = K \oplus \left\langle  d \right\rangle $  where $[d, d] = 0$, $[K, d] = [d, K] = \left\langle 0\right\rangle $, $K = Fa_1 \oplus Fa_2 \oplus \ldots \oplus Fa_n$, $[a_1, a_1] = a_2$, $[a_1, a_{j - 1}] = a_j$, $3 \leq j \leq n$, $[a_1, a_n] = 0$, $[a_m, a_k] = 0$  for all $m > 1$, $1 \leq k \leq n$.

(ii) $L = K + \left\langle  d \right\rangle $  where $\left\langle  d \right\rangle  = Fd \oplus F[d, d]$, $[d, [d, d]] = 0$, $[K, \left\langle  d \right\rangle ] = [\left\langle  d \right\rangle , K] = \left\langle 0\right\rangle $, $K = Fa_1 \oplus Fa_2 \oplus \ldots \oplus Fa_n$, $[a_1, a_1] = a_2$, $[a_1, a_{j - 1}] = a_j$, $3 \leq j \leq$ n, $[a_1, a_n] = 0$, $[a_m, a_k] = 0$  for all $m > 1$, $1 \leq k \leq n$.

(iii)  $L = K \oplus \left\langle  s \right\rangle $ where  $[s, s] = 0$, $K = Fa_1 \oplus Fa_2 \oplus \ldots \oplus Fa_n$, $[a_1, a_1] = a_2$, $[a_1, a_{j - 1}] = a_j$, $3\leq j \leq n$, $[a_1, a_n] = 0$, $[a_m, a_k] = 0$  for all $m > 1$, $1 \leq k \leq n$, 
\begin{align*}
	\left[ s, a_1\right]  &=& a_t + \gamma_{t + 1} a_{t + 1} + \gamma_{t + 2} a_{t + 2} +  \ldots + \gamma_{n - 1} a_{n - 1}   +  \gamma_{n}a_n,\\	
	\left[ s, a_2\right]  &=&     a_{t + 1}  + \gamma_{t + 1}a_{t + 2} + \ldots + \gamma_{n - 2}a_{n - 1}  +  \gamma_{n - 1}a_{n},\\
	\left[ s, a_3\right]  &=&  a_{t + 2}  + \ldots + \gamma_{n - 3}a_{n - 1}    +   \gamma_{n - 2}a_n,\\
	\ldots &\ldots \ldots \\
	\left[ s, a_{n - t + 1}\right]  &=&                                           a_n,
\end{align*}
$[s, a_j ] =  0$  for  $j > n - t + 1$, 
$[a_1, s] = \tau a_{n - t}$, $[a_j, s] = 0 $ whenever  $j \geq 2$  for some coefficients  $\gamma_{t + 1}, \ldots , \gamma_{n - 1}, \gamma_n,\tau \in F. $

\end{theorem}

The following result begins an examination of the non-nilpotent case.

\begin{theorem}\label{theorem B}
	Let $L$ be a non-nilpotent Leibniz algebra of finite dimension $n + 1 \geq 3$ over a field $F$. Suppose that  $L$  includes an ideal  $K$ of codimension  $1$. If  $K$  as a subalgebra is cyclic and nilpotent, then  $L$  has an element  $d$  such that  
	
	$L = K \oplus  Fd$  where $K = Fa_1 \oplus Fa_2 \oplus \ldots \oplus Fa_n$, $[a_1, a_1] = a_2$, $[a_1, a_{j - 1}] = a_j$, $3 \leq j \leq n$, $[a_1, a_n] = 0$, $[a_m, a_k] = 0$  for all $m > 1$, $1\leq k\leq n$; 

	$[a_1, d] = - a_1$, $[a_j, d] = 0$ whenever  $j \geq 2$; 
	\begin{align*}	
		\left[ d, a_1\right]  &=& a_1 + \gamma_2 a_2 +  \gamma_3 a_3 + \ldots + \gamma_{n - 1} a_{n - 1}  + \gamma_n a_n,\\
		\left[ d, a_2\right] &=&      2 a_2 + \gamma_2 a_3 + \ldots + \gamma_{n - 2}a_{n - 1}    +  \gamma_{n - 1}a_n,\\
		\left[ d, a_3\right]  &=&          3 a_3  + \ldots + \gamma_{n - 3} a_{n - 1}    + \gamma_{n - 2}a_n,\\
	 	\ldots &\ldots \ldots \\
		\left[ d, a_{n - 1}\right]  &=&             (n - 1)a_{n - 1}  + \gamma_2 a_n,\\
		\left[ d, a_n\right]  &=&                                               n a_n, 
	\end{align*}
	
	$[d, d] = - (\gamma_3 a_2 + \gamma_4 a_3 + \ldots + \gamma_n a_{n - 1}) + \delta_n a_n,$
for some coefficients  $\gamma_2$, $\gamma_3$, $\ldots$ , $\gamma_{n - 1}$, $\gamma_n$, $\delta_n$ $\in F$. 
	
\end{theorem}

The case of a field of characteristic $0$ is quite specific. The following result is devoted to it.

\begin{theorem}\label{theorem C}
Let $L$ be a non-nilpotent Leibniz algebra of finite dimension $n + 1 \geq 3$ over a field $F$. Suppose that  $\textbf{char}(F) = 0$,  and  $L$  includes an ideal $K$ of codimension  $1$. If  $K$  as a subalgebra is cyclic and nilpotent, then  $L$  satisfies the following conditions:

$L = K \oplus  \left\langle  s \right\rangle $  where  $[s, s] = 0$ (that is  $\left\langle  s \right\rangle  = Fs$),

$K$  has a basis  $\left\lbrace b_1, b_2,  \ldots  , b_n \right\rbrace $ such that  $[b_1, b_1] = b_2$, $[b_1, b_{j - 1}] = b_j$, $3 \leq j \leq n$, $[b_1, b_n] = 0$,    

$[b_m, b_k] = 0$  for all $m > 1$, $1 \leq k \leq n$; 

$[b_1, s] = - b_1$, $[b_j, s] = 0$ whenever  $j \geq 2$;
 
$[s, b_1] = b_1$, $[s, b_2] = 2b_2$, $[s, b_3] = 3b_3$,  \ldots , $[s, b_{n - 1}] = (n - 1)b_{n - 1}$, $[s, b_n] = nb_n$.  
\end{theorem}

\section{Nilpotent  Leibniz  algebras with maximal  cyclic  subalgebras}

A Leibniz algebra  $ L $  has a specific ideal. Denote by $ \textbf{Leib}(L) $ the subspace generated by the elements  $ [a, a] $, $ a \in L $. 

It is possible to show that  $ \textbf{Leib}(L) $  is an ideal of  $ L $,  and  if  $ H $  is an ideal of  $ L $  such that  $ L/H $  is a Lie algebra, then  $ \textbf{Leib}(L) \leq H $. 

The ideal  $ \textbf{Leib}(L) $  is called the \textit{Leibniz  kernel}  of algebra  $ L $. 

We note the following important property of the Leibniz kernel: 
$$
[[a, a], x] = 0 \text{ for arbitrary elements }   a, x  \in L.
$$

The \textit{left}  (respectively \textit{right}) \textit{center}  $ \zeta^{left}(L) $  (respectively  $ \zeta^{right}(L) $)  of a Leibniz algebra  $ L $  is defined by the rule 

$$
\zeta^{left}(L) = \left\lbrace  x \in L \,| \, [x, y] = 0 \text{ for each element } y \in L \right\rbrace.
$$

(respectively 
$$
\zeta^{right}(L) = \left\lbrace  x \in L \,| \, [y, x] = 0  \text{ for each element }  y \in L \right\rbrace ).
$$ 

It is not hard to prove that the left center of  $ L $  is an ideal, moreover,  $ \textbf{Leib}(L) \leq \zeta^{left}(L) $, so that  $ L/\zeta^{left}(L) $  is a Lie algebra. In general, the left and the right centers are different, moreover, the left center is an ideal, but it is not true for the right center: See the corresponding example in~\cite{KOP2016}.

The \textit{center}  $ \zeta(L)  $ of  $ L $  is defined by the rule: 
$$
\zeta(L) = \left\lbrace  x \in L \, | \, [x, y] = 0 = [y, x] \text{ for each element }  y  \in L \right\rbrace.
$$

The center is an ideal of  $ L $. 

Let  $ L $  be a Leibniz algebra. Define the lower central series of  $ L $  

$$
L = \gamma_1(L) \geq \gamma_2(L) \geq \ldots \geq \gamma_{\alpha}(L) \geq \gamma_{\alpha + 1}(L) \geq \ldots \gamma_{\delta}(L) 
$$
by the following rule: $ \gamma_1(L) = L $, $ \gamma_2(L) = [L, L] $, recursively  $ \gamma_{\alpha + 1}(L) = [L, \gamma_{\alpha}(L)] $  for all ordinals $ \alpha $,  and  $ \gamma_{\lambda}(L) = \bigcap_{ \mu < \lambda} \gamma_{\mu}(L) $  for the limit ordinals  $ \lambda $. The last term  $ \gamma_{\delta}(L) $  is called the \textit{lower hypocenter} of  $ L $. We have $ \gamma_{\delta}(L) = [L, \gamma_{\delta}(L)] $.   

If  $ \alpha = k  $ is a positive integer, then $  \gamma_k(L) = [L, [L, \ldots [ L , L]\ldots]] $  is the \textit{left normed commutator}  of  $ k $  copies of  $ L $. 

As usually, we say that a Leibniz algebra $ L $ is called \textit{nilpotent}, if there exists a positive integer $ k $  such that  $ \zeta_k(L) = \left\langle 0\right\rangle  $. More precisely, $ L $  is said to be \textit{nilpotent of nilpotency class} $ c $  if  $ \gamma_{c + 1}(L) = \left\langle 0\right\rangle  $, but  $ \gamma_c(L) \neq \left\langle 0\right\rangle  $. We denote the nilpotency class of  $ L $  by  $ \textbf{ncl}(L) $. 

Let  $ L $  be a Leibniz algebra over a field  $ F $  and  a  be an element of  $ L $. Put  
$$
\textbf{ln}_1(a) = a, \textbf{ln}_2(a) = [a, a],  \textbf{ln}_{k + 1}(a) = [a,  \textbf{ln}_k(a)], k \in \mathbb{N}.
$$
These elements are called the \textit{left normed commutators of the element}  $ a $.   

It will be useful to remind the following properties of these elements. These results have been proved in~\cite{CKSu2017}.

\begin{proposition}
Let  $ L $  be a Leibniz algebra  and  $ a $  be an element of  $ L $. Then the following assertion holds: 

(i)  $ [\textbf{ln}_k(a), \textbf{ln}_j(a)] = 0 $,  whenever  $ k > 1 $ and $ j \geq 1 $; in particular $ [\textbf{ln}_k(a), a] = 0  $. 

(ii)  Every non-zero product of $ k $ copies of an element $ a $  with any bracketing is coincides with  $ \textbf{ln}_k(a) $. 

(iii)  The cyclic subalgebra  $ \left\langle  a \right\rangle  $  is generated as a subspace by the elements  $ \textbf{ln}_k(a) $, $ k \in \mathbb{N} $. 

(iv)  The subalgebra  $ [\left\langle  a \right\rangle , \left\langle  a \right\rangle ] $  is generated as a subspace by the elements  $ \textbf{ln}_k(a) $ where  $ k \geq 2 $. 

(v) $ [\left\langle  a \right\rangle , \left\langle  a \right\rangle ] = \textbf{Leib}(\left\langle  a \right\rangle ) $. 

(vi)  The subalgebra  $ \gamma_k(\left\langle  a \right\rangle ) $  is generated as a subspace by the elements  $ \textbf{ln}_t(a) $ where  $ t \geq k $. 

(vii)  $ [\left\langle  a \right\rangle , \left\langle  a \right\rangle ]  \leq \zeta^{left}(L) $,   $ [\left\langle  a \right\rangle , \left\langle  a \right\rangle ] $  is an abelian subalgebra.  Moreover,  $ [[x, y], z] = 0 $  for all elements  $ x, y \in \left\langle  a \right\rangle  $  and an arbitrary element  $ z \in L $.   

\end{proposition}

We consider first  the  case  when a Leibniz algebra  $ L $  includes a maximal as a subalgebra ideal  $ K $ and this ideal $ K $  is cyclic as a Leibniz algebra. 

We will need the following simple result.

\begin{lemma}\label{lem 1}
	Let  $ L $  be a Leibniz algebra over a field  $ F $. If  $ L $  does not include proper non-zero subalgebras, then  $ \dim_F(L) = 1 $. In particular, $ L $  is a cyclic Lie algebra. 
\end{lemma}

\begin{proof}
	Indeed, if  $ L $  is a Lie algebra, the result is trivial. Suppose that  $ L $  is not a Lie algebra, then  $ \textbf{Leib}(L) = K $  is a non-zero subalgebra. Then  $ L = K $  is abelian, and $ \dim_F(L) = 1 $. 
\end{proof}

Let  $ L $  be a Leibniz algebra. A linear transformation  $ f $  of  $ L $  is called a \textit{derivation}, if  $ f([a, b]) = [f(a), b] + [a, f(b)]  $ for all  $ a, b \in L $. 

Denote by  $ \textbf{End}_F(L)  $ the set of all linear transformations of  $ L $. Then  $ \textbf{End}_F(L) $  is an associative algebra by the addition and multiplication of the transformations. As usual, $ \textbf{End}_F(L) $  is a Lie algebra by the operations  $ + $  and  $ [,] $ where  $ [f, g] =  f \circ g - g \circ f  $ for all  $ f, g \in \textbf{End}_F(L) $. It is possible to show that  $ \textbf{Der}(L) $  is a subalgebra of the Lie algebra  $ \textbf{End}_F(L) $.

Consider the mapping  $ \textbf{l}_a: L \longmapsto L $, defined by the rule  $ \textbf{l}_a(x) = [a, x] $, $ x \in L $. The mapping  $ \textbf{l}_a $  is a derivation of  $ L $  such that  $ \beta \textbf{l}_a = \textbf{l}_{\beta a}$, $\textbf{l}_a + \textbf{l}_b = \textbf{l}_{a + b} $  and  $ [\textbf{l}_a, \textbf{l}_b] = \textbf{l}_{[a, b]} $  for all elements  $ a, b \in L $, $ \beta \in F $. These equality show that the set  $ \left\lbrace  \textbf{l}_a \, | \, a \in L \right\rbrace  $  is a subalgebra of  $ \textbf{Der}(L) $.

\begin{lemma}
	Let  $ L $  be a Leibniz algebra over a field  $ F $,  and  $ f $  be a derivation of  $ L $. Then   $ f(\zeta^{left}(L) \leq  \zeta^{left}(L) $,  $ f(\zeta^{right}(L)) \leq \zeta^{right}(L) $  and  $ f(\zeta(L)) \leq \zeta(L) $.
\end{lemma}

\begin{proof}
	Let  $ x $  be an arbitrary element of  $ L $  and let  $ z \in \zeta^{left}(L) $. Then  $ [z, x] = 0 $. Since a derivation is a linear mapping, $ f([z, x]) = 0 $. On the other hand,  
	$$
	0 = f([z, x]) = [f(z), x] + [z, f(x)] = [f(z), x], 
	$$
	so that  $ f(z) \in \zeta^{left}(L) $.
	
	Let  $ z \in \zeta^{right}(L) $. Then  $ [x, z] = 0 $. Now we have 
	$$
	0 = f(0) = f([x, z]) = [f(x), z] + [x, f(z)] = [x, f(z)],
	$$
	so that  $ f(z) \in \zeta^{right}(L) $. The above proved inclusions imply that $ f(\zeta(L)) \leq \zeta(L) $.
\end{proof}

Define the upper central series 
$$
\left\langle 0\right\rangle  = \zeta_0(L) \leq \zeta_1(L) \leq \zeta_2(L) \leq \ldots \leq \zeta_{\alpha}(L) \leq \zeta_{\alpha + 1}(L) \leq \ldots \zeta_{\gamma}(L) = \zeta_{\infty}(L)
$$
of a Leibniz algebra  $ L $  by the following rule:  $ \zeta_1(L) = \zeta(L) $  is the center of  $ L $, recursively  $ \zeta_{\alpha + 1}(L)/\zeta_{\alpha}(L) = \zeta(L/\zeta_{\alpha}(L)) $  for all ordinals  $ \alpha $,  and  $ \zeta_{\lambda}(L) =  \bigcup_{\mu < \lambda} \zeta_{\mu}(L) $  for the limit ordinals  $ \lambda $. By definition, each term of this series is an ideal of  $ L $. The last term  $ \zeta_{\infty}(L) $  of this series is called the \textit{upper hypercenter of} $ L $. If  $ L = \zeta_{\infty}(L) $  then  $ L $  is called a \textit{hypercentral}  Leibniz algebra.  

\begin{corollary}
	\label{cor 1}
	Let  $ L $  be a Leibniz algebra over a field  $ F $,  and  $ f $  be a derivation of  $ L $. Then   $ f(\zeta_{\alpha}(L)) \leq  \zeta_{\alpha}(L) $  for every ordinal  $ \alpha $.
\end{corollary}

\begin{lemma}\label{lem 3}
	Let  $ L $  be a cyclic nilpotent Leibniz algebra over a field  $ F $, $ L = Fa_1 + \ldots + Fa_n $  where  $ [a_1, a_1] = a_2 $, $ [a_1, a_{j - 1}] = a_j $, $ 3 \leq j \leq n $, $ [a_1, a_n] = 0 $, $ [a_m, a_k] = 0 $  for all $ m > 1 $, $ 1\leq k \leq n $. If the linear mapping  $ f $  is a derivation of $  L $ then  
	\begin{align*}
		 f(a_1) &=& \gamma_1 a_1 + \gamma_2 a_2 + \gamma_3 a_3 + \ldots + \gamma_{n - 1} a_{n - 1} + \gamma_n a_n ,\\
		 f(a_2) &=&          2\gamma_1 a_2 + \gamma_2 a_3  + \ldots + \gamma_{n - 2}a_{n - 1} + \gamma_{n - 1}a_n ,\\
		 f(a_3) &=&                     3\gamma_1 a_3 + \ldots + \gamma_{n - 3}a_{n - 1}  + \gamma_{n - 2}a_n ,\\
		  &\ldots \\
		 f(a_{n - 1}) &=&                                  (n - 1)\gamma_{1}a_{n - 1}  + \gamma_2 a_n ,\\
		 f(a_n) &=&                                                           n\gamma_1 a_n .
	\end{align*}
\end{lemma}

\begin{proof}
	Put  $ Z_1 = Fa_n, Z_2 = Fa_n + Fa_{n - 1}, \ldots , Z_{n - 1} = Fa_n + \ldots + Fa_2, Z_n = L$, then 
	$$
	\left\langle 0\right\rangle  = Z_0 \leq Z_1 \leq Z_2 \leq  \ldots \leq Z_{n - 1} \leq Z_n = L
	$$
	is the upper central series of  $ L $. By Corollary~\ref{cor 1},  $ f(Z_j) \leq Z_j $, $ 1 \leq j \leq n $. It follows that 
	\begin{align*}
		f(a_1) &=& \gamma_{1\,1}a_1 + \gamma_{2\,1}a_2 + \gamma_{3\,1}a_3 + \ldots + \gamma_{n - 1\,1} a_{n - 1}   +  \gamma_{n\,1}a_n, \\
		f(a_2) &=&             \gamma_{2\,2}a_2 + \gamma_{3\,2}a_3  + \ldots + \gamma_{n - 1\, 2}a_{n - 1}   + \gamma_{n\,2}a_n, \\
		&\ldots&\\
		f(a_{n - 1}) &=&                                  \gamma_{n - 1 \, n - 1} a_{n - 1}    +  \gamma_{n \, n - 1} a_n,\\
		f(a_n) &=&                                               \gamma_{n \, n} a_n.  
	\end{align*}
	We have  
	$$
	\sum_{2 \leq j \leq n}  \gamma_{j \, 2} a_j = f(a_2) = f([a_1, a_1]) = [f(a_1), a_1] + [a_1, f(a_1)] =
	$$
	$$
	\left[ \sum_{1 \leq j \leq n} \gamma_{j \, 1} a_j, a_1\right]  + \left[ a_1, \sum_{1 \leq j \leq n} \gamma_{j \, 1} a_j \right]  = \gamma_{1 \, 1} a_2  + \sum_{1 \leq j \leq n-1}   \gamma_{j \, 1} a_{j + 1}.
	$$
	So, we obtain that
	$$
	\gamma_{2 \, 2} = 2\gamma_{1 \, 1}, \gamma_{3 \, 2}  = \gamma_{2 \, 1}, \gamma_{4 \, 2}  = \gamma_{3 \, 1}, \ldots , \gamma_{n - 1 \, 2}  = \gamma_{n - 2 \, 1}, \gamma_{n \, 2} = \gamma_{n - 1 \, 1}. 
	$$
	Further,
	$$
	\sum_{3\leq j\leq n}  \gamma_{j \, 3} a_j = f(a_3) = f([a_1, a_2]) = [f(a_1), a_2] + [a_1, f(a_2)] =
	$$
	$$
	\left[ \sum_{1\leq j \leq n} \gamma_{j \, 1} a_j, a_2 \right]  + \left[ a_1, \sum_{2\leq j\leq n} \gamma_{j \, 2} a_j \right]  = \gamma_{1 \, 1} a_3  + \sum_{2\leq j\leq n - 1} \gamma_{j \, 2} a_{j + 1}. 
	$$
	So, we obtain that 
	
	$\gamma_{3\, 3} = \gamma_{1 \, 1} + \gamma_{2 \, 2} = 3\gamma_{1 \, 1},$
	
	$\gamma_{4 \, 3} = \gamma_{3 \, 2} = \gamma_{2 \, 1}, $
	
	$\gamma_{5 \, 3} = \gamma_{4 \, 2} = \gamma_{3 \, 1}, $
	
	$\ldots $
	
	$ \gamma_{n - 1 \, 3}  = \gamma_{n - 2 \, 2} = \gamma_{n - 3 \, 1}, $
	
	$ \gamma_{n \, 3} = \gamma_{n-1 \, 2} =\gamma_{n - 2 \, 1}$.
	
	If  $ k > 2 $, then we have
	$$
	\sum_{ k \leq j \leq n}  \gamma_{j \, k} a_j = f(a_k) = f([a_1, a_{k - 1}]) = [f(a_1), a_{k - 1}] + [a_1, f(a_{k - 1})] =
	$$
	$$
	\left[ \sum_{1 \leq j \leq n} \gamma_{j \, 1} a_j, a_{k - 1}\right]  + \left[ a_1, \sum_{k-1 \leq j \leq n} \gamma_{j \, k-1} a_j \right]  = \gamma_{1 \, 1} a_k  + \sum_{k-1 \leq j \leq n-1}  \gamma_{j \, k - 1} a_{j + 1}, 
	$$
	and we obtain that

	$\gamma_{k \, k}  = \gamma_{1 \, 1} + \gamma_{k - 1 \, k - 1} = k\gamma_{1 \, 1},  $
	
	$\gamma_{k + 1 \, k}  = \gamma_{k \, k - 1} = \ldots = \gamma_{2 \, 1}, $
	
	$\gamma_{k + 2 \, k}  = \gamma_{k + 1 \, k - 1} = \ldots = \gamma_{3 \, 1}, $
	
	$\ldots $
	
	$\gamma_{n - 1 \, k}  = \gamma_{n - 2 \, k - 1} = \ldots = \gamma_{n - k \, 1},  $
	
	$\gamma_{n \, k}  = \gamma_{n-1 \, k-1}  = \dots = \gamma_{n - k + 1 \, 1} $. 
	
\end{proof}

Consider the mapping  $ \textbf{r}_a: L \longmapsto L $, defined by the rule  $ \textbf{r}_a(x) = [x, a] $, $ x \in L $. For every  $ x, y \in L $  and  $ \alpha \in F  $ we have $ \textbf{r}_a(x + y) = \textbf{r}_a(x) + \textbf{r}_a(y) $, $ \textbf{r}_a(\alpha x) = \alpha \textbf{r}_a(x) $ and 
$$
\textbf{r}_a([x, y]) = [[x, y], a] = [x, [y, a]] - [y, [x, a]] = [x, \textbf{r}_a(y)] - [y, \textbf{r}_a(x)]. 
$$
Also we have $ \beta \textbf{r}_a = \textbf{r}_{\beta a}$, and  $ \textbf{r}_a + \textbf{r}_b = \textbf{r}_{a + b} $  for all  $ a, b \in L $  and  $ \beta \in F $. 

In this connection, a linear mapping  $ g: L \longmapsto L $  is called the \textit{right derivation} if it satisfies the following condition  $ g([x, y]) = [x, g(y)] - [y, g(x)] $.

\begin{lemma}\label{lem 4}
	Let  $ L $  be a Leibniz algebra over a field  $ F $  and  $ g $  be a right derivation of  $ L $. Then   $ g(\zeta^{left}(L)) \leq  \zeta^{right}(L) $, in particular, $ g(\zeta(L)) \leq \zeta^{right}(L) $  and $ g(\textbf{Leib}(L)) = \left\langle 0\right\rangle  $.
\end{lemma}

\begin{proof}
	Let  $ x $  be an arbitrary element of  $ L $  and let  $ z \in \zeta^{left}(L) $. Then  $ [z, x] = 0 $. Since a right derivation is a linear mapping, $ g([z, x]) = 0 $. On the other hand,  
$$	
	0 = g([z, x]) = [z, g(x)] - [x, g(z)] = - [x, g(z)], 
$$	
	so that  $ g(z) \in \zeta^{right}(L) $.
	
	Let  $ x $  be an arbitrary element of  $ L $, we have  g$ ([x, x]) = [x, g(x)] - [x, g(x)] = 0 $. It follows that $  g(\textbf{Leib}(L)) = \left\langle 0\right\rangle  $.
	
\end{proof}

\begin{lemma}
	Let  $ L $  be a cyclic nilpotent Leibniz algebra over a field  $ F $, $ L = Fa_1 + \ldots + Fa_n $  where  $ [a_1, a_1] = a_2 $, $ [a_1, a_{j - 1}] = a_j $, $ 3 \leq j \leq n $, $ [a_1, a_n] = 0 $, $ [a_m, a_k] = 0 $  for all $ m > 1 $, $ 1 \leq k \leq n $. If a linear mapping  $ g $  is a right derivation of  $ L $, then  $ g(a_1) = \rho_1a_1 + \rho_2a_2 + \rho_3a_3 + \ldots + \rho_{n - 1} a_{n - 1}   + \rho_n a_n $, $ 	g(a_2) = g(a_3) = \ldots = g(a_n) = 0 $.      
	
\end{lemma}

\begin{proof}
	Indeed, by Lemma~\ref{lem 4} $ g(\textbf{Leib}(L)) = \left\langle 0\right\rangle  $. Since  $ \textbf{Leib}(L) = Fa_n + \ldots + Fa_2 $, we obtain that  $ g(a_2) = \ldots = g(a_n) = 0 $. 
\end{proof}

Let  $ L $  be a nilpotent Leibniz algebra of finite dimension, and  $ K $  be a maximal subalgebra of  $ L $. Then  $ K $  is an ideal of  $ L $~\cite[Theorem D]{KSuSe2018}. Suppose that  $ K $  is cyclic. Then  $ K = Fa_1 \oplus \ldots \oplus Fa_n $  where   
$$
[a_1, a_1] = a_2, [a_1, a_2] = a_3, \ldots , [a_1, a_{n - 1} ] = a_n, [a_1, a_n] = 0.
$$

Suppose first that  $ K $  is abelian. It is only possible if  $ \dim_F(K) = 1 $. In this case,  $ \dim_F(L) = 2 $. There are two non-isomorphic cases 
$$
L_1 = Fa + Fb, [a, a] =b, [b, a] = [a, b] = [b, b] = 0,
$$
and 
$$
L_2 = Fc + Fd, [c, c] = [c, d] = d,[d, c] = [d, d] = 0.
$$
(see, for example, the survey~\cite{KKPS2017}). In the first case,  $ L_1 $  is a nilpotent Leibniz algebra.

\begin{lemma} \label{lem 6}
	Let  $ L $  be a Leibniz algebra over a field  $ F $, and it of finite dimension. Suppose that  $ L $  includes an ideal  $ K $ of codimension  $ 1 $. If  $ K $  is a non-abelian and  $ \dim_F(K/\textbf{Leib}(K)) = 1 $, then  $ \textbf{Leib}(L) = \textbf{Leib}(K) $. 
\end{lemma}

\begin{proof}
	Since  $ K $  is non-abelian, $ K \neq \textbf{Leib}(L) $. This fact together with the obvious inclusion  $ \textbf{Leib}(K) \leq \textbf{Leib}(L) $  implies that either  $ \textbf{Leib}(K) = \textbf{Leib}(L) $  or  $ L = K + \textbf{Leib}(L) $. Consider the last case. Put  $ A = \textbf{Leib}(L) $, $ B = K \cap A $. Since  $ \dim_F(K/\textbf{Leib}(K)) = 1 $  and  $ \dim_F(L/K) = 1 $, we obtain that  
	$ \dim_F(L/\textbf{Leib}(K)) = 2 $. Since  $  \textbf{Leib}(K) \neq A $, $ \dim_F(L/A) = 1 $. We have $ L/B = (K/B) \oplus (A/B) $. Since  $ K $  and  $ A $  are ideals of  $ K $, $ [K, A] \leq B $. The isomorphisms  
	$$
	K/B = K/(K \cap A) \cong (K + A)/A = L/A 
	$$ 
	and  
	$$A/B = A/(K \cap A) \cong (K + A)/K = L/K
	$$
	show that   $ \dim_F(K/B) = \dim_F(A/B) = 1 $. In particular, $ K/B $  and  $ A/B $  are abelian, and, therefore,  $ L/B $  is abelian. In particular, $ L/B $  is a Lie algebra. But in this case.  $ B $  must includes  $ \textbf{Leib}(L) = A $, and we come to a contradiction. This contradiction shows that  $ \textbf{Leib}(L) = \textbf{Leib}(K) $. 
\end{proof}

\begin{corollary}
	Let  $ L $  be a nilpotent Leibniz algebra over a field  $ F $ having finite dimension. Suppose that  $ L $  includes a maximal cyclic subalgebra  $ K $. If  K  is non-abelian, then   $ \textbf{Leib}(L) = \textbf{Leib}(K) $. 
\end{corollary}

\begin{proof}
	Being a maximal subalgebra of a nilpotent Leibniz algebra, $ K $  is an ideal of  $ L $~\cite[Theorem D]{KSuSe2018}. Since  $ K $  is non-abelian, $ K \neq \textbf{Leib}(L) $. The fact that  $ K $  is cyclic implies that $ \dim_F(K/\textbf{Leib}(K)) = 1 $. Now we can apply Lemma~\ref{lem 6}. 
\end{proof}

\begin{corollary}
	Let  $ L $  be a nilpotent Leibniz algebra over a field  $ F $, having finite dimension. Suppose that  $ L $  includes a maximal cyclic subalgebra  $ K $. If  $ K $  is non-abelian, then  $ \textbf{ncl}(L) = \textbf{ncl}(K) $. 
\end{corollary}

\begin{proof}
	Put  $ A = \textbf{Leib}(L) $. By Lemma~\ref{lem 6} $ A = \textbf{Leib}(K) $. Lemma~\ref{lem 1} implies that  $ \dim_F(L/K) = 1 $. The equality  $ \dim_F(K/A) = 1 $  implies that  $ \dim_F(L/A) = 2 $. In other words, the factor-algebra  $ L/A $  is a Lie algebra of dimension $ 2 $. Note that a nilpotent Lie algebra of dimension  $ 2 $  is abelian, so that  $ L/A $  is abelian. It follows that  $ \gamma_2(L) = [L, L] \leq A $. On the other hand, $ \gamma_2(K) = A $, and, therefore,  $ \gamma_2(L) = A $. Since  $ A $  is an ideal of  $ L $, $ \gamma_3(L) = [L, \gamma_2(L)] = [L, A] \leq A $. Clearly, $ \gamma_3(K) \leq \gamma_3(L) $. We note that  $ \dim_F(A/\gamma_3(K)) = 1 $. Therefore if we assume that  $ \gamma_3(K) \neq \gamma_3(L) $, then  $ \gamma_3(L) = A = \gamma_2(L) $, and we obtain a contradiction with the fact that  $ L $  is nilpotent. Using the similar arguments we obtain that  $ \gamma_j(K) = \gamma_j(L) $  for all  $ j \in \left\lbrace  2, \ldots , n\right\rbrace  $. It follows that  $ \textbf{ncl}(L) = n $. 
\end{proof}

\begin{lemma}\label{lem 7}
	Let  $ L $  be a nilpotent Leibniz algebra over a field  $ F $ having finite dimension  $ n + 1 \geq 3 $. Suppose that  $ L $  includes a maximal cyclic subalgebra  $ K $. Then  $ L $  has an element  $ d $  such that  $ L = K \oplus Fd $  and  $ [K, d] = \left\langle 0\right\rangle  $. 
\end{lemma}

\begin{proof}
	Being a maximal subalgebra of a nilpotent Leibniz algebra, $ K $  is an ideal of  $ L $~\cite[Theorem D]{KSuSe2018}. Lemma~\ref{lem 1} implies that  $ \dim_F(L/K) = 1 $, so that,  $ L = K \oplus Fb $. Put  $ A = \textbf{Leib}(L) $. By Lemma~\ref{lem 6},  $ A = \textbf{Leib}(K) $, so that,  $ Fa_2 \oplus \ldots \oplus Fa_n  = \textbf{Leib}(L) $. Being a nilpotent Lie algebra of dimension $ 2 $  $ L/A $  is abelian. Then  
	$$
	[a_1, b] = \beta_2a_2 + \beta_3a_3 + \ldots + \beta_{n - 1} a_{n - 1} + \beta_n a_n   
	$$
	for some elements  $ \beta_2, \beta_3, \ldots , \beta_{n - 1}, \beta_n \in F $.
	
	Put  $ d = b - (\beta_2 a_1 + \beta_3 a_2 + \ldots + \beta_{n - 1} a_{n - 2} + \beta_n a_{n - 1}) $.  Then  
	$$
	[a_1, d] = [a_1, b - (\beta_2 a_1 + \beta_3 a_2 + \ldots + \beta_{n - 1} a_{n - 2} + \beta_n a_{n - 1})]=
	$$
	$$
	[a_1, b] - [a_1, \beta_2 a_1] - [a_1, \beta_3 a_2] -  \ldots - [a_1, \beta_{n - 1} a_{n - 2}] - [a_1, \beta_n a_{n - 1}] =
	$$
	$$
	\beta_2a_2 + \beta_3a_3 + \ldots + \beta_{n - 1} a_{n - 1} + \beta_n a_n - \beta_2 a_2 - \beta_3 a_3 - \ldots - \beta_{n - 1} a_{n - 1} - \beta_n a_n = 0.
	$$
	Furthermore, the inclusion  $ A \leq \zeta^{left}(L) $  implies that  $ [A, d] = \left\langle 0\right\rangle  $. Therefore,  $ [K, d] = \left\langle 0\right\rangle  $. 
\end{proof}

\begin{center}
		\textbf{Proof  of  Theorem~\ref{therem A}}
\end{center}

Being a maximal subalgebra of a nilpotent Leibniz algebra, $ K $  is an ideal of  $ L $ ~\cite[Theorem D]{KSuSe2018}. Lemma~\ref{lem 1} implies that  $ \dim_F(L/K) = 1 $. Being a nilpotent cyclic subalgebra, $ K $  has a basis  $ \left\lbrace  a_1, a_2, \ldots , a_n \right\rbrace  $  such that  
$$
[a_1, a_1] = a_2, [a_1, a_2] = a_3, . . . , [a_1, a_{n - 1} ] = a_n, [a_1, a_n] = 0, [a_m, a_k] = 0 
$$
for all $ m > 1, 1 \leq k \leq n $, ~\cite[Theorem 1.1]{CKSu2017}. Put  $ A = Fa_2 \oplus \ldots \oplus Fa_n  = \textbf{Leib}(K) $. By Lemma~\ref{lem 6} $ A = \textbf{Leib}(L) $, so that  $ Fa_2 \oplus \ldots \oplus Fa_n  = \textbf{Leib}(L) $. 

Put  $ Z_1 = Fa_n, Z_2 = Fa_n + Fa_{n - 1}, \ldots , Z_{n - 1} = Fa_n + \ldots + Fa_2, Z_n = K $. Then 
$$
\left\langle 0\right\rangle  = Z_0 \leq Z_1 \leq Z_2 \leq  \ldots \leq Z_{n - 1} \leq Z_n = K
$$
is the upper central series of  $ K $. 

If  $ w $  is an arbitrary element of  $ L $, then the mapping  $ \textbf{l}_w: K \longmapsto K $, which is defined by the rule  $ \textbf{l}_w(x) = [w, x] $, $ x \in K $, is a derivation of  $ K $. Then Corollary~\ref{cor 1} implies that  $ Z_j $  is a left ideal of  $ L $ for  $ j \in \left\lbrace  0, 1, \ldots , n - 1 \right\rbrace  $. On the other hand, $ Z_{n - 1} = \textbf{Leib}(L) \leq \zeta^{left}(L) $, so that $ Z_j $ is also a right ideal and therefore a two–sided ideal of  $ L $  for  $ j \in \left\lbrace  0, 1, \ldots , n - 1 \right\rbrace  $. 

Since  $ L $  is nilpotent, the fact that  $  Z_{j + 1}/Z_j  $ has dimension  $ 1 $  implies that                    $ Z_{j + 1}/Z_j \leq \zeta(L/Z_j) $  for  $ j \in \left\lbrace  0, 1, \ldots , n - 1 \right\rbrace  $ ~\cite[Theorem D]{KSuSe2018}. 

By Lemma~\ref{lem 7},  $ L = K \oplus Fd $  where  $ d $  is an element of  $ L $  such that  $ [K, d] = \left\langle 0\right\rangle  $. We can apply Lemma~\ref{lem 3} to the mapping  $ \textbf{l}_d $  and obtain that 
\begin{align*}
    \left[ d, a_1\right]  &=& \gamma_2a_2 + \gamma_3a_3 + \gamma_4a_4 + \ldots + \gamma_{n - 1} a_{n - 1} + \gamma_na_n, \\
	\left[ d, a_2\right]  &=&            \gamma_2a_3  + \gamma_3a_4 + \ldots + \gamma_{n - 2}a_{n - 1}  + \gamma_{n - 1}a_n, \\
	\left[ d, a_3\right]  &=&                        \gamma_2a_4  + \ldots + \gamma_{n - 3}a_{n - 1}  +  \gamma_{n - 2}a_n, \\
	\ldots \\
	\left[ d, a_{n - 1}\right] &=&                                       \gamma_2a_n,\\
	\left[ d, a_n\right] &=&  0. 
\end{align*}
Suppose that  $ \gamma_2 = \gamma_3 = \ldots = \gamma_{n - 1} =  \gamma_n = 0 $. Then  $ [d, K] = \left\langle 0\right\rangle  $. Here we obtain two cases: $ [d, d] = 0 $  and  $ [d, d] \neq 0 $. In the first case,  $ \left\langle  d \right\rangle  = Fd $, and  $ L = K \oplus \left\langle  d \right\rangle  $ where  $ K $  and  $ \left\langle  d \right\rangle = Fd$  are ideals. In the second case,  $ \left\langle  d \right\rangle  = Fd \oplus F[d, d] $  is a nilpotent subalgebra of dimension $ 2 $, moreover, $ \left\langle  d \right\rangle  $  is an ideal of $ L $ and $ L = K + \left\langle  d \right\rangle  $, $ K \cap \left\langle  d \right\rangle  = \zeta(L) $. In particular, if  $ n = 3 $, then  $ L $  is the sum of two ideals, that are cyclic subalgebras of dimension $ 2 $, and their intersection is the center of  $ L $. In this algebra every subalgebra is an ideal. Thus, we can consider this Leibniz algebra as an analog of the quaternion group. 

Suppose now that there are non-zero elements among the elements  $ \gamma_2, \gamma_3, \ldots , \gamma_{n - 1}, \gamma_n $. Let  $ t $  be the least index such that  $ \gamma_t \neq 0 $. Then 
\begin{align*}
	[d, a_1] &=& \gamma_ta_t +  \gamma_{t + 1} a_{t + 1} + \gamma_{t + 2} a_{t + 2} + \ldots + \gamma_{n - 1} a_{n - 1}   + \gamma_na_n,\\
	[d, a_2] &=&             \gamma_t a_{t + 1}  + \gamma_{t + 1}a_{t + 2} + \ldots + \gamma_{n - 2}a_{n - 1}  + \gamma_{n - 1}a_n,\\
	[d, a_3] &=&                             \gamma_t a_{t + 2}  + \ldots + \gamma_{n - 3}a_{n - 1}  +  \gamma_{n - 2}a_n,\\
	\ldots\\
	[d, a_{n - t + 1}] &=&                                                   \gamma_ta_n,\\
\end{align*}

$[d, a_j] =  0$  for  $j > n - t + 1. $

Without loss of generality we may assume that  $ \gamma_t = 1 $. Otherwise, because $ [K, \gamma_t^{-1}d] = \gamma_t^{-1}[K, d] = \left\langle 0\right\rangle $, instead of the element  $ d $  we will consider the element  $ \gamma_t^{-1}d $.

The equality  $ A = \textbf{Leib}(L) $  implies that  $ [d, d] \in A $, so that  $ [d, d] = \sum_{ 2 \leq j \leq n} \delta_ja_j $. Since  $ [a_1, d] = 0 $, $ [a_1, [d, d]] = 0 $, we have  
$$
[a_1, [d, d]] = [[a_1, d], d] + [d, [a_1, d]] = 0.
$$
Furthermore,
$$
[a_1, [d, d]] = [a_1, \delta_2a_2 + \delta_3a_3 + \ldots + \delta_{n - 1} a_{n - 1} + \delta_na_n] = \delta_2a_3 + \delta_3a_4 + \ldots + \delta_{n - 1}a_n.
$$
The elements  $ a_3, a_4,  \ldots , a_n $  are linearly independent, and therefore the equality    
$$
\delta_2a_3 + \delta_3a_4 + \ldots + \delta_{n - 1}a_n = 0  
$$
implies  $ \delta_2 = \delta_3 = \ldots = \delta_{n - 1} = 0 $. Hence $ [d, d] = \delta_na_n $.

If we put  $ x = \delta_n a_{n - t + 1} $, then we obtain  
$$
[d - x, d - x] = [d, d] - [d, x] - [x, d] + [x, x] = [d, d] - [d, x] = \delta_na_n - [d, \delta_n a_{n - t + 1}] =
$$
$$
\delta_na_n - \delta_n [d, a_{n - t + 1}] =  \delta_na_n - \delta_na_n  = 0.
$$

Put  $ s = d - x $. Then  $ [s, s] = 0 $, so that  $ \left\langle  s \right\rangle  = Fs $. Furthermore,                                                                                    
$$[a_1, s] = [a_1, d - x] = [a_1, d] - [a_1, x] = 0 - [a_1, \delta_n a_{n - t + 1}] = - \delta_n a_{n - t} = \tau a_{n - t},$$
$$
[a_j, s] = [a_j, d - x] = [a_j, d] - [a_j, \delta_n a_{n - t + 1},] = 0  \mbox{  whenever  }   j \geq 2.
$$
Since  $ n - t + 1 \geq 2 $, 
$$
[s, a_1] = [d - x, a_1] = [d, a_1] - [x, a_1] = [d, a_1] - [\delta_n a_{n - t + 1}, a_1] = [d, a_1],
$$
$$
[s, a_j] = [d - x, a_1] = [d, a_j] - [x, a_j] = [d, a_j] - [\delta_n a_{n - t + 1}, a_j] = [d, a_j]  
$$
whenever  $ j \geq 2 $.

\section{Non-nilpotent  Leibniz  algebras, having  maximal  cyclic  subalgebras}

\begin{center}
	\textbf{Proof  of  Theorem~\ref{theorem B}}
\end{center}

Since  $ \dim_F(L/K) = 1 $,  $ L = K \oplus Fb $  for some element  $ b $. Being a nilpotent cyclic subalgebra, $ K $  has a basis  $ \left\lbrace  a_1, a_2, \ldots , a_n \right\rbrace  $  such that  for all $ m > 1 $, $ 1 \leq k \leq n $,
$$
[a_1, a_1] = a_2, [a_1, a_2] = a_3, \ldots , [a_1, a_{n - 1} ] = a_n, [a_1, a_n] = 0, [a_m, a_k] = 0  
$$
\cite[Theorem 1.1]{CKSu2017}. Put  $ A = Fa_2 \oplus \ldots \oplus Fa_n  = \textbf{Leib}(K) $. By Lemma~\ref{lem 6} $ A = \textbf{Leib}(L) $, so that  $ Fa_2 \oplus \ldots \oplus Fa_n  = \textbf{Leib}(L) $. We have  $ [K, K] = A $. If we suppose that  $ L/A $  is nilpotent, then  $ L $  is itself nilpotent~\cite[Theorem~3.1]{RCGHHZ2014}. It follows that 
$$
[b, a_1] = \beta_1a_1 + \beta_2a_2 + \beta_3a_3 + \beta_4a_4 + \ldots + \beta_{n - 1} a_{n - 1} + \beta_na_n
$$
for some coefficients  $ \beta_1, \beta_2, \ldots , \beta_{n - 1},  \beta_n \in F $, where $ \beta_1 \neq 0 $. Without loss of generality we may assume that   $ \beta_1 = 1 $. Indeed, if $ \beta_1 \neq 1 $, then we consider the element  $ b_1 = \beta_1^{-1}b $. Clearly $ L = K \oplus Fb_1 $. 

Since $ L/A $  is a Lie algebra, we obtain 
$$
[a_1, b] = - a_1 + \sigma_2a_2 + \sigma_3a_3 + \ldots + \sigma_{n - 1} a_{n - 1} + \sigma_na_n
$$
for some elements  $ \sigma_2, \sigma_3, \ldots, \sigma_{n - 1}, \sigma_n \in F $. 

Put  $ d = b - (\sigma_2a_1 + \sigma_3a_2 + \ldots + \sigma_{n - 1} a_{n - 2} + \sigma_na_{n - 1}) $.  Then  
$$
[a_1, d] = [a_1, b - (\sigma_2a_1 + \sigma_3a_2 + \ldots + \sigma_{n - 1} a_{n - 2} + \sigma_na_{n - 1})] =
$$
$$
[a_1, b] - [a_1, \sigma_2a_1] - [a_1, \sigma_3a_2] - \ldots - [a_1, \sigma_{n - 1} a_{n - 2}] - [a_1, \sigma_n a_{n - 1}] =
$$
$$
- a_1 + \sigma_2a_2 +  \ldots + \sigma_{n - 1} a_{n - 1} + \sigma_na_n - \sigma_2a_2  - \ldots - \sigma_{n - 1} a_{n - 1} - \sigma_na_n = -a_1.
$$
Furthermore, the inclusion  $ A \leq \zeta^{left}(L) $  implies that  $ [A, d] = \left\langle 0\right\rangle  $. 

If  $ w $  is an arbitrary element of  $ L $, then the mapping  $ \textbf{l}_w: L \longmapsto L $, defined by the rule  $ \textbf{l}_w(x) = [w, x] $, $ x \in K $, is a derivation of  $ K $. Then Corollary~\ref{cor 1} implies that  $ Z_j $  is a left ideal of  $ L $ for  $ j \in \left\lbrace  0, 1, \ldots , n - 1 \right\rbrace  $. On the other hand, $ Z_{n - 1} = \textbf{Leib}(L) \leq \zeta^{left}(L) $, so that  $ Z_j $  is also a right ideal and therefore two-sided ideal of  $ L $  for  $ j \in \left\lbrace  0, 1, . . . , n - 1 \right\rbrace  $. 

We can apply Lemma~\ref{lem 3} to the mapping  $ l_d $  and obtain that 
\begin{align*}
	[d, a_1] &=& a_1 + \gamma_2a_2 + \gamma_3a_3 + \ldots + \gamma_{n - 1} a_{n - 1} + \gamma_n a_n,\\
	[d, a_2] &=&        2a_2 + \gamma_2a_3  + \ldots + \gamma_{n - 2}a_{n - 1}  + \gamma_{n - 1}a_n,\\
	[d, a_3] &=&                    3a_3  + \ldots + \gamma_{n - 3}a_{n - 1}  + \gamma_{n - 2}a_n,\\
	\ldots\\
	[d, a_{n - 1}] &=&                                    (n - 1)a_{n - 1}  + \gamma_2a_n,\\
	[d, a_n] &=&                                                                na_n. 
\end{align*}
The equality  $ A = \textbf{Leib}(L) $  implies that  $ [d, d] \in A $, so that 
$$
[d, d] = \delta_2a_2 + \delta_3a_3 + \ldots + \delta_{n - 1} a_{n - 1} + \delta_na_n.
$$
We have  
$$
[a_1, [d, d]] = [[a_1, d], d] + [d, [a_1, d]]. 
$$
Furthermore
$$
[a_1, [d, d]] = [a_1, \delta_2a_2 + \delta_3a_3+ \ldots + \delta_{n - 1} a_{n - 1} + \delta_na_n] = \delta_2a_3 + \delta_3a_4 + \ldots + \delta_{n - 1}a_n.
$$
$$
[[a_1, d], d] = [- a_1, d] = a_1,
$$
$$
[d, [a_1, d_]] = [d, -a_1] = -[d, a_1] = -(a_1 + \gamma_2a_2 + \gamma_3a_3 +  \ldots + \gamma_{n-1} a_{n - 1} + \gamma_na_n).
$$
Thus, we obtain
$$
\delta_2a_3 + \delta_3a_4 + \ldots + \delta_{n - 1}a_n = a_1 - (a_1 + \gamma_2a_2 + \gamma_3a_3 + \gamma_4a_4 + \ldots + \gamma_{n - 1} a_{n - 1} + \gamma_na_n) = 
$$
$$
- \gamma_2a_2 - \gamma_3a_3 - \gamma_4a_4 - \ldots - \gamma_{n - 1} a_{n - 1} - \gamma_na_n.
$$
It follows that $  \gamma_2 = 0, \delta_2 = - \gamma_3, \delta_3 = - \gamma_4, \ldots , \delta_{n - 2} = - \gamma_{n - 1}, \delta_{n - 1} = - \gamma_n $. Thus
$$
[d, d] = - (\gamma_3a_2 + \gamma_4a_3 + \ldots + \gamma_na_{n - 1}) + \delta_na_n.
$$

\begin{center}
	\textbf{Proof of Theorem~\ref{theorem C}}
\end{center}

Being a nilpotent cyclic subalgebra, $ K $  has a basis  $ \left\lbrace  a_1, a_2, \ldots , a_n \right\rbrace   $ such that  
$$
[a_1, a_1] = a_2, [a_1, a_2] = a_3, . . . , [a_1, a_{n - 1} ] = a_n, [a_1, a_n] = 0, [a_m, a_k] = 0  
$$
for all $ m > 1, 1 \leq k \leq n $,~\cite[Theorem 1.1]{CKSu2017}. Put  $ A = Fa_2 \oplus \ldots \oplus Fa_n  = \textbf{Leib}(K) $. By Lemma~\ref{lem 6},  $ A = \textbf{Leib}(L) $, so that  $ Fa_2 \oplus \ldots \oplus Fa_n  = \textbf{Leib}(L) $. We have also  $ [K, K] = A $. Using Theorem~\ref{theorem B} we obtain that  $ L $  contains an element  $ d $  such that  $ L = K \oplus Fd $  and  $ [a_1, d] = - a_1 $, $ [a_j, d] = 0 $ whenever  $ j \geq 2 $. Theorem~\ref{theorem B} shows also that 
\begin{align*}
	[d, a_1] &=& a_1 + \gamma_2a_2 + \gamma_3a_3 + \ldots + \gamma_{n - 1} a_{n - 1} + \gamma_n a_n,\\
	[d, a_2] &=&        2a_2 + \gamma_2a_3  + \ldots + \gamma_{n - 2}a_{n - 1}  + \gamma_{n - 1}a_n,\\
	[d, a_3] &=&                    3a_3  + \ldots + \gamma_{n - 3}a_{n - 1}  + \gamma_{n - 2}a_n,\\
	\ldots\\
	[d, a_{n - 1}] &=&                                    (n - 1)a_{n - 1}  + \gamma_2a_n,\\
	[d, a_n] &=&                                                                na_n. 
\end{align*}

We chose an element  $ x \in A $  such that $ [d, d] = [d, x] $.  Put  $ x = \sum_{ 2 \leq j \leq n}  \lambda_j a_j $, then  
 $$ [d, x] = \left[ d, \sum_{ 2 \leq j \leq n}  \lambda_j a_j\right]  = \sum_{ 2 \leq j \leq n}  \lambda_j [d, a_j] = $$
 $$\lambda_2(2a_2 + \gamma_2a_3  + \gamma_3a_4  + \gamma_4a_5  + \ldots + \gamma_{n - 2}a_{n - 1} + \gamma_{n - 1}a_n) + $$
$$\lambda_3(3a_3  + \gamma_2a_4  + \gamma_3a_5 + \ldots + \gamma_{n - 3}a_{n - 1} + \gamma_{n - 2}a_n) + $$
$$\lambda_4(4a_4  + \gamma_2a_5  + \gamma_3a_6 + \ldots + \gamma_{n - 2}a_{n - 1}  + \gamma_{n- 3}a_n) + \ldots +$$
$$\lambda_{n - 1} ((n - 1)a_{n - 1}  + \gamma_2a_n) + n\lambda_n a_n =$$
$$2\lambda_2a_2  + (\lambda_2\gamma_2 + 3\lambda_3)a_3 + (\lambda_2\gamma_3 + \lambda_3\gamma_2 + 4\lambda_4)a_4   + (\lambda_2\gamma_4 + \lambda_3\gamma_3 + \lambda_4\gamma_2 + 5\lambda_5)a_5  + \ldots +$$
$$(\lambda_2\gamma_{n - 2} + \lambda_3\gamma_{n - 3} + \lambda_4\gamma_{n - 4}  + \ldots + \lambda_{n - 2}\gamma_2 + (n - 1)\lambda_{n - 1})a_{n - 1} + $$
$$(\lambda_2\gamma_{n - 1} + \lambda_3\gamma_{n - 2}  + \lambda_4\gamma_{n - 3} + \ldots + \lambda_{n - 1}\lambda_2 + n\lambda_n)a_n. $$

By Theorem~\ref{theorem B},  $ [d, d] = - (\gamma_3a_2 + \gamma_4a_3 + \ldots +  \gamma_na_{n - 1}) + \delta_na_n $, and therefore from the equality $ [d, d] = [d, x] $  we obtain the system of linear equations 
\begin{align*}
	-\gamma_3 &=   	       2\lambda_2\\
	-\gamma_4 &=    \gamma_2   \lambda_2+ 		        3\lambda_3 \\
	-\gamma_5 &=    \gamma_3   \lambda_2 +       \gamma_2 \lambda_3+                4\lambda_4\\
	\ldots &\ldots\ldots\\ 
	-\gamma_n &= \gamma_{n - 2}\lambda_2  + \gamma_{n - 3}\lambda_3  + \gamma_{n - 4}\lambda_4  + \ldots + \gamma_2\lambda_{n - 2} +  (n - 1)\lambda_{n - 1}\\
	\delta_n  &= \gamma_{n - 1}\lambda_2  + \gamma_{n - 2}\lambda_3  + \gamma_{n - 3}\lambda_4  + \ldots + \gamma_3\lambda_{n - 2} + \gamma_2\lambda_{n - 1} +  n \lambda_n
\end{align*}
The last system has a nonsingular matrix, therefore it has a unique solution $ \lambda_j , 2 \leq j \leq n $. Thus, we have shown that an element $ x $ satisfying the equality $[d, d] = [d, x] $ exists.
 
We have now
$$
[d - x, d - x] = [d, d] - [d, x] - [x, d] + [x, x] = [d, d] - [d, x] = 0.
$$
Since  $ x \in A = \zeta^{left}(L) $, $ [x, L] = \left\langle 0\right\rangle  $, therefore  
$$
[d - x, a_j] = [d, a_j] - [x, a_j] = [d, a_j]  \mbox{ for all }  j, 1 \leq j \leq n.
$$

Put  $ s = d - x $, then  $ [s, s] = 0 $ and $ [s, a_j] =[d, a_j] $  for all  $ j, 1 \leq j \leq n $. 

If  $ b = a_1 + \lambda_2a_3 + \lambda_3a_4 + \ldots + \lambda_{n - 2}a_{n - 1}  + \lambda_{n - 1}a_n $, then 
$$
[b, s] = [a_1 + \lambda_2a_3 + \lambda_3a_4 + \ldots + \lambda_{n - 2}a_{n - 1}  + \lambda_{n - 1}a_n, d - x] =
$$
$$
[a_1, d] - [a_1, x] = -a_1 - [a_1, \lambda_2a_2 + \lambda_3a_3 + \lambda_4a_4 + \ldots + \lambda_{n - 1} a_{n - 1} + \lambda_n a_n] =
$$
$$
-a_1 - \lambda_2a_3 - \lambda_3a_4 - \lambda_4a_5 - \ldots - \lambda_{n - 2}a_{n - 1} - \lambda_{n - 1}a_n = - b.
$$
Furthermore, $ [s, [b, s]] = [[s, b], s] + [b, [s, s]] = [[s, b], s] $. On the other hand,  $ [s, [b, s]] = [s, - b] = -[s, b] $, and we obtain that  $ [[s, b], s] = -[s, b] $. We have 
$$
[s, b] = [d - x, b] = [d, b] = a_1 + \nu_2a_2 + \nu_3a_3 + \ldots + \nu_{n - 1}a_{n - 1}  + \nu_na_n.
$$
for some coefficients  $ \nu_2, \nu_3, \ldots , \nu_{n - 1}, \nu_n \in F $. We note that the element  
$$
c = \nu_2a_2 + \nu_3a_3 + \ldots + \nu_{n - 1}a_{n - 1}  + \nu_na_n
$$
belong to  $ A = \textbf{Leib}(L) \leq \zeta^{left}(L) $. Now we obtain 
$$
[[s, b], s] = [a_1 + \nu_2a_2 + \nu_3a_3 + \ldots + \nu_{n - 1}a_{n - 1}  + \nu_na_n, d - x] = [a_1 + c, d - x] =
$$
$$
[a_1 + c, d] - [a_1 + c, x] = - a_1  - [a_1 + c, x] = - a_1  - [a_1, x] =
$$
$$
- a_1  - [a_1, \lambda_2a_2 + \lambda_3a_3 + \ldots + \lambda_{n - 1} a_{n - 1} +  \lambda_n a_n] =
$$
$$
- a_1  - \lambda_2a_3 - \lambda_3a_4 - \ldots - \lambda_{n - 2} a_{n - 1} - \lambda_{n - 1}a_n = - b. 
$$
It follows that  $ [s, b] =  - [[s, b], s] = b $. 

Put  $ b_1 = b, b_2 = [b_1, b_1], b_j = [b_1, b_{j - 1}], 3 \leq j \leq n $, then 
$$
[s, b_2] = [s, [b_1, b_1]] = [[s, b_1], b_1] + [b_1, [s, b_1]] = [b_1, b_1] + [b_1, b_1] = 2[b_1, b_1] = 2b_2,  
$$
$$ 
[s, b_3] = [s, [b_1, b_2]] = [[s, b_1], b_2] + [b_1, [s, b_2]] = [b_1, b_2] + [b_1, 2b_2] = 3[b_1, b_2] = 3b_3. 
$$
Using an ordinary induction, we obtain that  $ [s, b_j] = jb_j $  for all  $ j \leq n $.

We have 
$$b_2 = [a_1 + \lambda_2a_3 + \lambda_3a_4 + \ldots + \lambda_{n - 1}a_n, a_1 + \lambda_2a_3 + \lambda_3a_4 + \ldots +  \lambda_{n - 1}a_n] =
$$
$$ [a_1, a_1 + \lambda_2a_3 + \lambda_3a_4 + \ldots + \lambda_{n - 2}a_{n - 1} + \lambda_{n - 1}a_n] = a_2 + \lambda_2a_4 + \lambda_3a_5 + \ldots + \lambda_{n - 2}a_n,
$$

$$
b_3 = [a_1 + \lambda_2a_3 + \lambda_3a_4 + \ldots + \lambda_{n - 2}a_{n - 1} + \lambda_{n - 1}a_n, a_2 + \lambda_2a_4 + \lambda_3a_5 + \ldots + \lambda_{n - 2}a_n] =
$$
$$ 
[a_1, a_2 + \lambda_2a_4 + \lambda_3a_5 + \ldots + \lambda_{n - 2}a_n] = a_3 + \lambda_2a_5 + \lambda_3a_6 + \ldots + \lambda_{n - 3}a_n, \ldots  , 
$$

$$
b_{n - 3} = [b_1, b_{n - 4}] = [a_1 + \lambda_2a_3 + \lambda_3a_4 + \ldots + \lambda_{n - 1}a_n, a_{n - 4} + \lambda_2a_{n - 2} + \lambda_3a_{n - 1} + \lambda_4a_n] = 
$$
$$
[a_1, a_{n - 4} + \lambda_2a_{n - 2} + \lambda_3a_{n - 1} +\lambda_4a_n] = a_{n - 3} + \lambda_2a_{n - 1} + \lambda_3a_n,  
$$

$$
b_{n - 2} = [b_1, b_{n - 3}] = [a_1 + \lambda_2a_3 + \lambda_3a_4 + \ldots + \lambda_{n - 1}a_n, a_{n - 3} + \lambda_2a_{n - 1} + \lambda_3a_n]=
$$
$$
[a_1, a_{n - 3} + \lambda_2a_{n - 1} + \lambda_3a_n] = a_{n - 2} + \lambda_2a_n, 
$$

$$
b_{n - 1} = [b_1, b_{n - 2}] = [a_1 + \lambda_2a_3 + \lambda_3a_4 + \ldots + \lambda_{n - 1}a_n, a_{n - 2} + \lambda_2a_n] = 
$$
$$
[a_1, a_{n - 2} + \lambda_2a_n] = a_{n - 1},  
$$

$$
b_n = [b_1, b_{n - 1}] = [a_1 + \lambda_2a_3 + \lambda_3a_4 + \ldots + \lambda_{n - 1}a_n, a_{n - 1}] = 
$$
$$
[a_1, a_{n - 1}] = a_n, [b_1, b_n] = [b_1, a_n] = 0.
$$

Thus, we see that  transition from the system of elements $ \left\lbrace  a_1, a_2, \ldots , a_{n - 1}, a_n \right\rbrace  $  to   $ \left\lbrace  b_1, b_2, \ldots, b_{n - 1}, b_n \right\rbrace  $  is provided by the matrix. 

$$\begin{pmatrix}
1 & 0 & \lambda_2 & \lambda_3 & \lambda_4 & \cdots &  \lambda_{n-3} & \lambda_{n-2} & \lambda_{n-1}\\
0 & 1 & 0 & \lambda_2 & \lambda_3 & \cdots &  \lambda_{n-4} & \lambda_{n-3} & \lambda_{n-2}\\
0 & 0 & 1 &      0    & \lambda_2 & \cdots &  \lambda_{n-5} & \lambda_{n-4} & \lambda_{n-3}\\
\cdots & \cdots & \cdots & \cdots & \cdots & \cdots & \cdots& \cdots        & \cdots\\
0 & 0 & 0 &      0    & 0         & \cdots &       1        &       0       & \lambda_{2}\\
0 & 0 & 0 &      0    & 0         & \cdots &       0        &       1       & 0\\
0 & 0 & 0 &      0    & 0         & \cdots &       0        &       0       & 1\\
\end{pmatrix}$$
This matrix is nonsingular, which means that the elements $  b_1, b_2, \ldots, b_{n - 1}, b_n  $  form a basis of $ K $.

Kurdachenko, L. A. and Chupordia V.A., Department of Algebra and Geometry,
Oles Honchar Dnipro National University, Gagarin prospect 72, Dnipro 10, 49010, Ukraine.
e-mail: lkurdachenko@i.ua, vchupordia@gmail.com

Subbotin, I.Ya., Mathematics Department , National University,
5245 Pacific Concourse Drive, Los Angeles, CA,90045, U.S.A.
e-mail: isubboti@nu.edu

\end{document}